\documentclass[10pt]{amsart}
\usepackage{amsmath,amscd}
\usepackage{amsbsy}
\usepackage{amssymb}
\usepackage{amscd,amsthm}
\usepackage[all,cmtip]{xy}
\usepackage[russian,english]{babel}
\usepackage[utf8]{inputenc}
\usepackage{xcolor}
\usepackage{dsfont}

\usepackage{hyperref}

\renewcommand{\epsilon}{\varepsilon}

\renewcommand{\ell}{x}

\newtheorem{thm}{Theorem}\numberwithin{thm}{section}
\newtheorem{lem}[thm]{Lemma}

\newtheorem{rema}[thm]{Remark}

\newtheorem*{con2}{Conjecture}

\newtheorem*{rema2}{Remark}

\begin{document}
	\begin{center}
		\huge{Products of factorials which are products of factorials}\\[1cm] 
	\end{center}
	\begin{center}
		\large{Sa$\mathrm{\check{s}}$a Novakovi$\mathrm{\acute{c}}$}\\[0,5cm]
		{\small February 2026}\\[0,2cm]
	\end{center}
	{\small \textbf{Abstract}. 
		In this note, we look at the diophantine equation
		$$
		\prod_{i=1}^ta_i!=\prod_{j=1}^sn_i!, \quad n_1\geq \cdots \geq n_s\geq 2 \quad  \textnormal{and}\quad n_1>a_1\geq a_2\geq\cdots \geq a_t\geq2.
		$$\\
		\noindent
	Let $s<t$. Under the (explicit) abc conjecture, we show that it has only finitely many nontrivial solutions in a certain subset of $\mathbb{N}^{t+s}$ of positive asymptotic density.}
	\begin{center}
		\tableofcontents
	\end{center}
	\section{Introduction}
	\noindent
	It is a prominent and long-standing unsolved problem whether the equation
	$$
	\prod_{i=1}^ta_i!=n! \quad n>a_1\geq a_2\geq \cdots \geq a_t\geq 2
	$$
	has only finitely many solutions. For a positive integer $n$ we write $P(n)$ for the largest prime factor of $n$. Erd\H{o}s and Graham \cite{EG} made the observation that if
	$$
	\frac{P(n(n+1))}{\mathrm{log}(n)}\longrightarrow \infty \quad \textnormal{as} \quad n\rightarrow \infty
	$$
	then the equation would have only finitely many solutions. A trivial solution has $a_1=n-1$. Thus, $n=a_2!\cdots a_t!$ and consequently, there are infinitely many trivial solutions. The Sur$\mathrm{\acute{a}}$nyi--Hickerson conjecture predicts that $16!=14!5!2!$ is the largest nontrivial solution and Nair and Shorey \cite{NS} showed that under the Baker's explicit abc conjecture, the equation has indeed only the nontrivial solutions 
	$$
	7!3!3!2!=9!, 7!6!=10!, 7!5!3!=10!, 14!5!2!=16!, 15!2!^4=16!
	$$ 
	The classical abc conjecture barely implies $P(n(n+1))\geq (1+o(1))\mathrm{log}(n)$ as $n\rightarrow \infty$, which is weaker the the estimate of Erd\H{o}s and Graham. However, Luca \cite{LUC} showed that the classical abc conjecture implies that 
	$$
	\prod_{i=1}^ta_i!=n! \quad n>a_1\geq a_2\geq \cdots \geq a_t\geq 2
	$$
	has only finitely many solutions. This result was sharpend by Luca, Saradha and Shorey \cite{LUC2} under Baker's explicit abc conjecture and explicit bounds for $a_2$ were obtained. In \cite{WT2} Takeda studied a generalization of the Sur$\mathrm{\acute{a}}$nyi--Hickerson conjecture and established all nontivial solutions to
	$$
	\prod_{i=1}^ta_i!=\prod_{j=1}^sn_i!, \quad n_1\geq \cdots \geq n_s\geq 2 \quad  \textnormal{and}\quad a_1\geq a_2\geq\cdots \geq a_t\geq2,
	$$
	with $a_i\neq n_j$ for any $i,j$ in the case of small $t$ and $s=2$, provided $\mathrm{max}\{n_1, a_1\}$ is bounded. Question 4.3 in \emph{loc. cite.} asks whether that equation has finitely many nontrivial solutions. We give a partial confirmative answer under the (explicit) abc conjecture. The proof is an adaptation of the proof given by Luca, Saradha and Shorey \cite{LUC2} and also avoids Mateev's result on linear forms in logarithms. Interestingly, it can be proved unconditionally, that the equation has finitely many trivial solutions if the factorials were replaced by an analogue notion of factorials which can be defined over number fields (see \cite{WTT} and \cite{WT}).\\
	\noindent
	Since we want to apply the (explicit) abc conjecture, we recall its statement. For a non-zero integer $a$, let $N(a)$ be the \emph{algebraic radical}, namely $N(a)=\underset{p\mid a}{\prod}{p}$.
	\begin{con2}[classical abc conjecture]
		For any $\epsilon >0$ there is a constant $K(\epsilon)$ depending only on $\epsilon$ such that whenever $a,b$ and $c$  are three coprime and non-zero integers with $a+b=c$, then 
		\begin{eqnarray*}
			c<K(\epsilon)N(abc)^{1+\epsilon}
		\end{eqnarray*}
		holds.
	\end{con2} 
	\noindent
	In 1975 Baker \cite{BAK} gave an explicit version of the abc conjecture. Then Laishram and Shorey \cite{LS}, Theorem 1 showed that Baker's version implies that
	$$
	c<N(abc)^{7/4}.
	$$ 
	We shall use this inequality in place of classical abc or Baker's explicit abc conjecture in proving our main result. We simply say \emph{explicit abc conjecture}. We now specify our problem. In this note, we focus on the equation
	$$
	\prod_{i=1}^ta_i!=\prod_{j=1}^sn_i!
	$$ 
	for $a_1\geq a_2\geq \cdots \geq a_t\geq 2$ and $n_1\geq n_2\geq \cdots \geq n_s\geq 2$ with $a_i\neq n_j$ for any $i,j$. We assume $s<t$. We say that a solution $(a_1,...,a_t,n_1,...,n_s)$ is trivial if there exists a pair $(i,j)$ such that $|a_i-n_j|=1$. We first deal with the case $s=2$. As the examples and calculations in the work of Takeda \cite{WT2} suggest, there are many nontrivial solutions with $n_1>a_1$ and $n_2>a_i$ for some $2\leq i\leq t$. So we may assume $n_1>a_1$ and $n_2>a_i$ for some $2\leq i\leq t$. Let $k_1=n_1-a_1$ and $k_2=n_2-a_i$. Note that $k_j\geq 2$ to exclude trivial solutions. Denote by $m_1=a_1+1$ and by $m_2=a_i+1$. Then we can rewrite the diophantine equation and obtain
	\begin{center}
	\begin{eqnarray}
	\underset{l\neq i,1 }{\prod}a_l!=\Delta(m_1,k_1)\Delta(m_2,k_2),
\end{eqnarray}
	\end{center}
	\noindent
	where $\Delta(m_i,k_i)=m_i(m_i+1)\cdots (m_i+k_i-1)$. Notice that $m_i+k_i-1=n_i$. Let $c>0$ be an arbitrary but fixed positive integer. For $2\leq i\leq t$, we define
	$$
	\mathcal{N}(c)_i=\{\textnormal{nontrivial solutions to (1)}\mid n_1>a_1, n_2>a_i, \quad k_2/k_1\leq c\}.
	$$
	Then let
	$$
	\bigcup_{i=2}^t\mathcal{N}(c)_i=\mathcal{N}(c).
	$$
	\begin{thm}
	The explicit abc conjecture implies that (1) has only finitely many solutions $(a_1,...,a_t,n_1,n_2)\in \mathcal{N}(c)_i$. Hence the explicit abc conjecture implies that
	$$
	\prod_{i=1}^ta_i!=n_1!n_2!
	$$
	has finitely many solutions $(a_1,...,a_t,n_1,n_2)\in \mathcal{N}(c)$.
	\end{thm}
	\noindent
	The proof of Theorem 1.1 shows that the argument can be generalized to the case where $2<s<t$. Let $I=\{2,...,t\}$ and denote by $I_{s-1}$ the set of all subsets $\{i_2,...,i_{s}\}\subset I$ consisting of $s-1$ distinct elements of $I$. Assume that for any $j=2,...,s$ there exists an $i_j$ such that $n_j>a_{i_j}$. Then the diophantine equations simplifies to
	\begin{eqnarray}
	\prod_{i\notin\{1,i_2,...,i_{s}\}}a_i!= \prod_{j=1}^s\Delta(m_j,k_j),
	\end{eqnarray}
	where $m_j=a_{i_j}+1$ and $k_j=n_j-a_{i_j}$ for $j=2,...,s$. Notice that $m_1=a_1+1$ and $k_1=n_1-a_1$. Let $\mathrm{max}\{k_2,...,k_s\}=k$ and let $c>0$ be an arbitrary but fixed positive integer.  We define the following set
	$$
	\mathcal{N}(c)_{\{i_2,...,i_{s}\}}=\{\textnormal{nontrivial solutions to (2)}\mid n_1>a_1, n_j>a_{i_j} \quad \textnormal{and} \quad k/k_1\leq c \}.
	$$
	Then we define 
	$$
	\mathcal{N}(c)=\bigcup_{\{i_2,...,i_{s}\}\in I_{s-1}}\mathcal{N}(c)_{\{i_2,...,i_{s}\}}.
	$$
	\begin{thm}
		The explicit abc conjecture implies that (2) has only finitely many solutions $(a_1,...,a_t,n_1,...,n_s)\in \mathcal{N}(c)_{\{i_2,...,i_{s}\}}$.  Hence the explicit abc conjecture implies that
		$$
		\prod_{i=1}^ta_i!=\prod_{j=1}^sn_i!
		$$
		has finitely many solutions $(a_1,...,a_t,n_1,...,n_s)\in \mathcal{N}(c)$.
	\end{thm}
	\begin{rema2}
	\textnormal{One could make Theorems 1.1 and 1.2 explicit by keeping track of all the constants and coefficients in the estimates that appear in the proofs. But that is not our intention.}
	\end{rema2}
	\noindent
	Let $d(\mathcal{N}(c))$ denote the asymptotic density of $\mathcal{N}(c)$ in $\mathbb{N}^{t+s}$. We consider equation (1). Since $\mathcal{N}(c)_i\not\subset \mathcal{N}(c)_j$ for $i\neq j$, we get $d(\mathcal{N}(c))\geq d(\mathcal{N}(c)_i)$. The next result shows that $\mathcal{N}(c)$ is large in the sense that its asymptotic density is positive. 
	\begin{thm}
		We have $d(\mathcal{N}(c)_i)>0$.
	\end{thm}
	\noindent
	Diophantine equations of the form $a_1!\cdots a_t!=f(x,y)$ have also been studied quite frequently in the literature and we neither want to give a complete list of reference here nor we want to recall the main results with respect to such equations. We refer, for instance, to \cite{EO}, \cite{WT}, \cite{NA1}, \cite{NO} or \cite{NOO} and references therein. We just want to mention that showing that such an equation has finitely many integer solutions is still a widely open problem. In particular, if $f(x,y)$ is an arbitrary reducible and non-homogeneous polynomial. Theorem 1.2 from above essentially treats one special case. Fix a positive integer $k_i>0$ and let $P(x_i)=x_i(x_i+1)\cdots (x_i+k_i-1)$ be an integer polynomial. Now consider the equation
	\begin{eqnarray}
\prod_{i=1}^ta_i!=\prod_{j=1}^sP(x_i)
	\end{eqnarray}
	Again, let $\mathrm{max}\{k_2,...,k_s\}=k$ and $c>0$ be an arbitrary but fixed positive integer. Define
	$$
	\mathcal{L}=\{\textnormal{integer solutions to (3)} \mid x_1\geq \cdots \geq x_s, x_1>a_1\geq \cdots \geq a_t, k/k_1\leq c\}
	$$
	\begin{thm}
		The explicit abc conjecture implies that equation (3) has only finitely many integer solutions $(a_1,...,a_t,x_1,...,x_s)\in \mathcal{L}$.  
	\end{thm}
	\noindent
	The proof is in fact the same as for Theorem 1.2 since the statement of Theorem 1.4 is just a reformulation in terms of polynomials. 
	\section{Preliminaries}
	We first state a few well known facts from number theory. As mentioned in the introduction, $P(n)$ denotes the greatest prime factor of $n$. We put $P(1)=1$. Let $\Delta(x,k)=x(x+1)\cdots (x+k-1)$. Erd\H{o}s \cite{ERD} proved that there exists a large number $\kappa$ such that
	$$
	P(\Delta(x,k))>\frac{2}{7}k\mathrm{log}(k)\quad \textnormal{for} \quad k\geq \kappa
	$$
	whenever $x,x+1,...,x+k-1$ are all composite integers. In fact, Erd\H{o}s proved the result for an unknown constant in place of $2/7$. A proof with $2/7$ can be found in \cite{LUC2}, Lemma 7. Next, we recall some well known estimates. Fix a positive real number $\nu>1$. Let $\theta(\nu)=\sum_{p\leq \nu}\mathrm{log}(p)$.
	\begin{lem}
		We have
		\begin{itemize}
			\item [(i)] $\theta(\nu)<1.00008\nu$ \textnormal{for} $\nu>1$.
			\item[(ii)] $\underset{p\leq \nu}{\sum}\frac{\mathrm{log}(p)}{p}<\mathrm{log}(\nu)$ \textnormal{for} $\nu>1$. 
		\end{itemize}
		\end{lem}
		\begin{proof}
			For a proof of (i) see \cite{DU} and \cite{DU1}. For (i) see \cite{RO}. 
		\end{proof}
		\begin{lem}
			Consider equation (1). 
			\begin{itemize}
				\item [(i)] None of the terms in $\Delta(m_1,k_1)$ is a prime.
				\item[(ii)] Let $n_2>a_i$ and set $a=\underset{l\neq i,1}{\mathrm{max}}\{a_l\}$. Then 
				$$
				a\cdot\mathrm{log}(a)-a\leq \mathrm{log}(a!)\leq (k_1+k_2)\mathrm{log}(2m_1).
				$$
			\end{itemize}
		\end{lem}
		\begin{proof}
			Assume $m_1+d=p$ is prime for some $0\leq d<k_1$. Then 
			$$
			a_1+1\leq m_1+d=p\leq a.
			$$
			But this contradicts $a_1\geq a_2\geq \cdots \geq a_t$. We now prove (ii). Suppose $m_1<k_1$. By Bertrand's postulate, there is a prime in $((m_1+k_1-1)/2, m_1+k_1-1)$. But this interval is contained in $[m_1,m_1+k_1-1]$, which is a contradiction, since none of the terms in $\Delta(m_1,k_1)$ is a prime according to (i). This shows $m_1\geq k_1$. By (1), we obtain
			$$
			a\cdot \mathrm{log}(a)-a\leq \mathrm{log}(a!)\leq \mathrm{log}(m_1+k_1)^{k_1+k_2}\leq (k_1+k_2)\mathrm{log}(2m_1).
			$$
		\end{proof}
		\noindent
		In the same way one can prove:
		\begin{lem}
				Consider equation (2). 
			\begin{itemize}
				\item [(i)] None of the terms in $\Delta(m_1,k_1)$ is a prime.
				\item[(ii)] Let $a=\underset{i\notin\{1,i_2,...,i_s\}}{\mathrm{max}}\{a_i\}$. Then 
				$$
				a\cdot\mathrm{log}(a)-a\leq \mathrm{log}(a!)\leq (k_1+\cdots+ k_s)\mathrm{log}(2m_1).
				$$
			\end{itemize}
		\end{lem}
	
	\section{Proof of Theorems 1.1 and 1.2}
	\noindent
	We follow the line of argument as in \cite{LUC2} with some minor changes. We give the proof for $s=2$. So we consider equation (1). Without loss of generality, we may assume $a_2>n_2>a_3$. Thus $a=a_2$. We want to prove that $a_2$ is bounded. By Lemma 2.2 (i), no term in $\Delta(m_1,k_1)$ is a prime. Observe that primes $>k_1$ divide at most one term of $\Delta(m_1,k_1)$. For every prime $\leq k_1$, we delete the term in which it appears to the maximum power. Further, all primes dividing $\Delta(m_1,k_1)$ are $\leq a_2$. Thus
	\begin{equation*}
		\begin{split}
	\prod_{i=0}^{k_1-1}N(m_1+i)\leq \left(\underset{k_1\leq p\leq a_2}{\prod} p\right)\left( \underset{p<k_1}{\prod}p^{\lfloor k_1/p \rfloor}\right) \\
	\leq \mathrm{exp}\left( \underset{k_1\leq p\leq a_2}{\sum}\mathrm{log}(p)+k_1\underset{p<k_1}{\sum}\frac{\mathrm{log}(p)}{p}\right).
\end{split}
\end{equation*}
By Lemma 2.1, we get
$$
\prod_{i=0}^{k_1-1}N(m_1+i)\leq \mathrm{exp}(1.00008a_2+k_1\cdot \mathrm{log}(k_1)).
$$
Choose $m_1+j_1$ and $m_1+j_2$ such that $N(m_1+j_1)\leq N(m_1+j_2)$ are the smallest among the $N(m_1+i)$ for $0\leq i<k_1$. Then
$$
N(m_1+j_2)\leq \left(\prod_{i=0}^{k_1-1}N(m_1+i) \right)^{1/(k_1-1)}\leq \mathrm{exp}\left(\frac{1.00008a_2}{k_1-1}+\frac{k_1\cdot \mathrm{log}(k_1)}{k_1-1} \right).
$$
Now consider 
$$
\frac{m_1+j_1}{d}-\frac{m_1+j_2}{d}=\frac{j_1-j_2}{d},
$$	
where $d=\mathrm{gcd}(m_1+j_1,j_1-j_2)$. Applying the explicit abc conjecture yields
$$
\frac{m_1}{d}\leq \left( N(m_1+j_1)N(m_1+j_2)\cdot \frac{|j_1-j_2|}{d}\right)^{7/4}.
$$
Hence
$$
\mathrm{log}(m_1)\leq \frac{7}{4}\left(\frac{2.00016a_2}{k_1-1}+\frac{2k_1\cdot \mathrm{log}(k_1)}{k_1-1}+\mathrm{log}(k_1) \right)
$$
and therefore
\begin{equation}
k_1\cdot \mathrm{log}(m_1)\leq \frac{7}{4}\left(\frac{k_1\cdot 2.00016a_2}{k_1-1}+\frac{2k_1^2\cdot \mathrm{log}(k_1)}{k_1-1}+k_1\cdot \mathrm{log}(k_1) \right).
\end{equation}
Since $k_1\geq 2 $ by assumption, we have $k_1/(k_1-1)\leq 2$. Furthermore, because $P(\Delta(m_1,k_1))>\frac{2}{7}k_1\cdot \mathrm{log}(k_1)$ for $k_1\geq \kappa$ , we conclude 
\begin{equation}
a_2>\frac{2}{7}k_1\cdot \mathrm{log}(k_1) \quad \textnormal{for} \quad k_1\geq \kappa.
\end{equation}
In the remaining part of the proof, we have to distinguish two cases, namely $k_2\leq k_1$ and $k_1<k_2$.\\
\noindent
\emph{the case $k_2\leq k_1$}:\\
\noindent
We assume $k_1\geq \kappa$. From (4) and (5), we conclude that there exists a positive integer constant $c_1$ such that
\begin{equation}
k_1\cdot \mathrm{log}(k_1)\leq c_1\cdot a_2.
\end{equation}
Now use the assumption $k_2\leq k_1$, inequality (6) and Lemma 2.2 (ii) to obtain
$$
a_2\cdot \mathrm{log}(a_2)-a_2\leq (k_1+k_2)\mathrm{log}(2m_1)\leq 2k_1\cdot \mathrm{log}(2m_1)\leq c_2a_2, 
$$
where $c_2$ is a positive integer constant. Resolving to $\mathrm{log}(a_2)$ yields
$$
\mathrm{log}(a_2)\leq c_3
$$
for some constant $c_3>0$. Now consider the case $k_1<\kappa$. We may assume $a_2>c_4\kappa\cdot \mathrm{log}(\kappa)$ for a suitable constant $c_4>0$, since otherwise the result follows. Then (4) yields
$$
k_1\cdot \mathrm{log}(m_1)\leq c_5a_2+c_6k_1\cdot \mathrm{log}(k_1),
$$
where $c_5,c_6$ are positive constants. Again, we use Lemma 2.2 (ii) to get
$$
a_2\cdot \mathrm{log}(a_2)\leq c_7a_2+c_8k_1\cdot\mathrm{log}(k_1)
$$
and hence
$$
\mathrm{log}(a_2)\leq c_7+\frac{c_8}{c_4}\frac{k_1\cdot \mathrm{log}(k_1)}{\kappa\cdot \mathrm{log}(\kappa)}\leq c_7+\frac{c_8}{c_4}.
$$
This shows that $a_2$ is bounded when $k_2\leq k_1$.\\
\noindent
\emph{the case $k_1<k_2$}:\\
\noindent
Notice that $k_1<k_2$ and $k_2\leq ck_1$ by assumption. Since the boundedness of $k_2/k_1$ is equivalent to the boundedness of $(k_2-1)/(k_1-1)$ if $k_1>1$, we have 
$$
N(m_1+j_2)^{k_2-1}\leq N(m_1+j_2)^{b(k_1-1)}=(N(m_1+j_2)^{k_1-1})^b
$$
for a suitable $b>0$. And since 
$$
N(m_1+j_2)^{k_1-1}\leq \prod_{i=0}^{k_1-1}N(m_1+i),
$$
we get
$$
N(m_1+j_2)^{k_2-1}\leq \left(\prod_{i=0}^{k_1-1}N(m_1+j_2) \right)^b\leq \mathrm{exp}(1.00008\cdot ba_2+bk_1\cdot \mathrm{log}(k_1)).
$$
Therefore
$$
N(m_1+j_2)\leq \mathrm{exp}\left(\frac{1.00008\cdot ba_2}{k_2-1}+\frac{bk_1\cdot \mathrm{log}(k_1)}{k_2-1} \right).
$$
Again, we consider
$$
\frac{m_1+j_1}{d}-\frac{m_1+j_2}{d}=\frac{j_1-j_2}{d},
$$	
where $d=\mathrm{gcd}(m_1+j_1,j_1-j_2)$ and apply the explicit abc conjecture. This gives us
$$
\mathrm{log}(m_1)\leq \frac{7}{4}\left(\frac{2.00016\cdot ba_2}{k_2-1}+\frac{2bk_1\cdot \mathrm{log}(k_1)}{k_2-1}+\mathrm{log}(k_1) \right)
$$
Now we multiply by $k_2$ and find
$$
k_2\cdot \mathrm{log}(m_1)\leq \frac{7}{4}\left(\frac{k_2\cdot 2.00016\cdot ba_2}{k_2-1}+\frac{2k_2k_1b\cdot \mathrm{log}(k_1)}{k_2-1}+ck_1\cdot \mathrm{log}(k_1) \right).
$$
Assuming $k_1\geq \kappa$ and because $k_2/(k_2-1)\leq 2$, we get
$$
k_2\cdot \mathrm{log}(m_1)\leq c_9a_2.
$$
for a suitable positive constant $c_9>0$. Now Lemma 2.2 (ii) states
$$
a_2\cdot \mathrm{log}(a_2)-a_2\leq (k_1+k_2)\mathrm{log}(2m_1)\leq 2k_2\cdot \mathrm{log}(2m_1).
$$
Pluging in $k_2\cdot \mathrm{log}(m_1)\leq c_9a_2$ shows that $a_2$ is bounded. Now if $k_1<\kappa$ proceed as in the proof for the case $k_2\leq k_1$. The details are left to the reader. This completes the proof for the case $s=2$. Theorem 1.2 can be proved in the same way. One just has to use $a$ instead of $a_2$ and $k$ instead of $k_2$. 
\begin{rema}
	\textnormal{We are not sure whether Theorem 1.1 can be proved under some weaker condition, for example Szpiro's conjecture. It is also not clear to us if one can find a nontrivial subset $\mathcal{N}\subset \mathbb{N}^{t+s}$ with density $d(\mathcal{N})>0$, for which one can show, unconditionally, that it contains only finitely many solutions.}
\end{rema}
\section{Proof of Theorem 1.3}
\noindent
Without loss of generality, we consider $A:=\mathcal{N}(c)_2$. To keep notation simple, we prove the assertion for $s=2$ and $t=3$. Now
$$
d(A)=\underset{N\rightarrow \infty}{\mathrm{lim}}\frac{|A\cap \{1,...,N\}^5|}{N^5}=\underset{N\rightarrow \infty}{\mathrm{lim}}\frac{1}{N^5}\underset{n_1,n_2,a_1,a_2,a_3}{\sum}\mathds{1}_{A}(\frac{n_1}{N},\frac{n_2}{N},\frac{a_1}{N}, \frac{a_2}{N},\frac{a_3}{N}), 
$$
where $\mathds{1}$ denotes the indicator function of $A$. We can use integrals to calculate that limit. We set $x_i=n_i/N$ and $y_j=a_j/N$. Then we have the restrictions $x_1\geq x_2$, $y_1\geq y_2\geq y_3$, $x_1>y_1$, $x_2>y_2$ and $(x_2-y_2)\leq c(x_1-y_1)$. The density is calculated as
$$
d(A)=V_1-V_2,
$$
where $V_1$ and $V_2$ are given by
$$
V_1=\int_0^1dx_1\int_0^{x_1}dy_1\int_0^{y_1}dy_2\int_{y_2}^{x_1}dx_2\int_0^{y_2}dy_3=\frac{1}{60}
$$
and
$$
V_2=\int_0^1dx_1\int_0^{x_1}dy_1\int_0^{x_1-c(x_1-y_1)}dy_2\int_{y_2+c(x_1-y_1)}^{x_1}dx_2\int_0^{y_2}dy_3=\frac{1}{120c}.
$$
Here $V_1$ is the volume that is given without the restriction $k_2/k_1\leq c$, or equivalently $(x_2-y_2)\leq c(x_1-y_1)$. Notice that $x_2\leq y_2+c(x_1-y_1)$ implies that we have to cut of the region where $x_2>y_2+c(x_1-y_1)$. And since $x_2\leq x_1$, we get $y_2+c(x_1-y_1)<x_1$. Moreover, this also means that $y_2<x_1-c(x_1-y_1)$. Similar arguments can also be applied for arbitrary $t>0$ to show that the density is indeed $>0$.

	\vspace{0.3cm}
	\noindent
	{\tiny HOCHSCHULE FRESENIUS UNIVERSITY OF APPLIED SCIENCES 40476 D\"USSELDORF, GERMANY.}\\
	E-mail adress: sasa.novakovic@hs-fresenius.de\\
	

\begin{thebibliography}{999}
		\bibitem{BAK} A. Baker, Experiments on the abc-conjecture. Publ. Math. Debrecen 65 (2004), 253-260. 
		\bibitem{DU} P. Dusart, Autour de la fonction qui compte le nombre de nombre premiers. Ph.D thesis, Universit\'e de Limoges (1998)
		\bibitem{DU1} P. Dusart, In\'egalit\'es explicites pour $\psi(X), \theta(X), \pi(X)$ et les nombres premiers. C.R. Math. Acad. Sci. Soc. R. Can. 21 (1999), 53-59.
		\bibitem{EO} P. Erd\H{o}s and R. Obl\'ath, \"Uber diophantische Gleichungen der Form $n!=x^p \pm y^p$ und $n!\pm m!= x^p$. Acta Szeged. 8 (1937), 241-255.
		\bibitem{ERD} P. Erd\H{o}s, On consecutive integers. Nieuw. Arch. Wiskd. 3 (1955), 124-128.
		\bibitem{EG} P. Erd\H{o}s and R.L. Graham, Old and new problems and resolutions in combinatorial number theory, Monography No. 28 L'Enseignement Math. Geneve (1980). 
		
		\bibitem{LS} S. Laishram and T.N. Shorey, Baker's explicit abc-conjecture and applications. Acta Arith. 155 (2012), 419-429.
		\bibitem{LUC} F. Luca, On factorials which are products of factorials. Math. Proc. Camb. Phil. Soc. (2007), 143-533.
		\bibitem{LUC2} F. Luca, N. Saradha and T.N. Shorey, Squares and factorials in products of factorials. Monatsh. Math. (2014), 385-400.  
		\bibitem{NA1} A. Makki Naciri, On the variant $Q(n!)=P(x)$ of the Brocard--Ramanujan Diophantine equation. Ramanujan J. 65 (2024), 1791-1798.
		\bibitem{NS} S. G. Nair and T.N. Shorey, Lower bounds for the greatest prime factor product of consecutive positive integers. J. Number Theory 159 (2016), 307-328.
		\bibitem{NO} S. Novakovi\'c, A note on some polynomial-factorial Diophantine equations. Glasnik Matematicki 60 (2025), 21-38.
		\bibitem{NOO} S. Novakovi\'c, On a generalization of the Brocard--Ramanujan Diophantine equation, arXiv:2602.09678
		\bibitem{RO} J.B. Rosser and L. Schoenfeld, Approximate formulas for some functions of prime numbers. Illinois J. of Math. 6 (1962), 64-94.
		\bibitem{WTT} W. Takeda, Finiteness of trivial solutions of factorial products yielding a factorial over number fields. Acta Arith. 190 (2019), 395-401.
		\bibitem{WT} W. Takeda, On the finiteness of solutions for polynomial-factorial Diophantine equations. Forum Math. 33 (2021), 361-374.
		\bibitem{WT2} W. Takeda, Product of Factorials Equal Another Product of Factorials. Bull. Iranian Math. Soc. 50 (2024), 1-23.
		
	\end{thebibliography}
	\end{document}